\documentclass[10pt]{amsart}

\usepackage[colorlinks]{hyperref}
\hypersetup{
citecolor = {blue}
}
\usepackage{color,graphicx,shortvrb}
\usepackage[latin 1]{inputenc}
\usepackage{amssymb}
\usepackage[active]{srcltx} 

\usepackage{enumerate}

\newtheorem{theorem}{Theorem}[section]

\newtheorem{corollary}[theorem]{Corollary}

\newtheorem{definition}[theorem]{Definition}
\newtheorem{lemma}[theorem]{Lemma}

\newtheorem{proposition}[theorem]{Proposition}
\newtheorem{remark}[theorem]{Remark}

\def\W{W_0^{1,p(x)}(\Omega)}

\def\RR{{\mathbb{R}}}
\def\NN{{\mathbb{N}}}
\def\11{\textbf{$1$}}

\def\w{\mathcal{W}(\Omega)}
\def\hlambda{{\hat{\lambda}}}
\def\G{{\widehat{G}}_\lambda}

\DeclareGraphicsExtensions{.jpg,.pdf,.png,.eps}

\usepackage{enumerate}

\begin{document}

\title[A  concave-convex problem with a variable operator]{A  concave-convex problem with a variable operator}

\author[A. Molino and J. D. Rossi]{Alexis Molino and Julio D. Rossi}

\address{
A. Molino: Departamento de An\'alisis Matem\'atico, Campus Fuen\-te\-nue\-va S/N, Uni\-ver\-si\-dad de Granada 18071 - Granada, Spain.
{\tt  amolino@ugr.es}}

\address{
J. D. Rossi: Departamento de Matem\'{a}tica, FCEyN, Universidad de Buenos Aires, Ciudad
Universitaria, Pab~1~(1428), Buenos Aires, Argentina.
{\tt jrossi@dm.uba.ar}}

\date{}

\begin{abstract}
We study the following elliptic problem $-A(u) = \lambda u^q$ with Dirichlet boundary conditions, where $A(u) (x) = \Delta u (x) \chi_{D_1} (x)+ \Delta_p u(x) \chi_{D_2}(x)$ is the Laplacian
in one part of the domain, $D_1$, and the $p-$Laplacian (with $p>2$) in the rest of the domain,
$D_2 $. We show that this problem exhibits a concave-convex nature for $1<q<p-1$. In fact, we prove that there exists a positive value $\lambda^*$  such that
the problem has no positive solution for $\lambda > \lambda^*$ and a minimal positive solution for 
$0<\lambda < \lambda^*$. If in addition we assume that $p$ is subcritical, that is, $p<2N/(N-2)$ then there are
at least two positive solutions for almost every $0<\lambda < \lambda^*$, the first one (that exists for all 
$0<\lambda < \lambda^*$) is obtained minimizing a suitable
functional and the second one (that is proven to exist for almost every $0<\lambda < \lambda^*$) comes from an appropriate (and delicate) mountain pass argument.
\end{abstract}

\maketitle

\begin{center}
{\it To Ireneo Peral a great mathematician and friend in his 70th birthday.}
\end{center}

 \thispagestyle{empty}
 
\section{Introduction}
Given a smooth bounded domain $\Omega$ we split it into two smooth subdomains 
$$
\overline \Omega = \overline {D_1 \cup D_2}, \qquad D_1\cap D_2 = \emptyset
$$
(we assume that both $D_1$ and $D_2$ are Lipschitz).
We call $\Gamma$ the interface inside $\Omega$,
$$
\Gamma = \partial D_1 \cap \Omega = \partial D_2 \cap \Omega,
$$
 and we assume that $\Gamma$ is a smooth surface with finite $(N-1)$ dimensional Hausdorff measure.

  For a fixed $p>2$ we consider the operator  which acts as the Laplacian in the region $D_1$ and as the 
$p$-Laplacian in the region $D_2$. To be more precise, we consider equations of the form
\[
-\Delta u=f(u), \hbox{ in } D_1 \qquad \hbox{ and }\qquad -\Delta_p u=f(u), \hbox{ in } D_2,
\]
with a Dirichlet boundary condition, $u=0$ on $\partial \Omega$, a suitable continuity condition on $\Gamma$ and a power nonlinearity $f$.

  Note that this problem can also be rewritten involving a variable exponent ope\-rator, a $p(x)$-Laplacian, with a discontinuous exponent $p(x)$. That is, we deal with
\begin{equation*}
\left\{\begin{array}{cl}
-\Delta_{p(x)}u=f(u), & \mbox{ in } \Omega,
\\
u=0, & \mbox{ on } \partial \Omega,
\end{array} \right.
\end{equation*}
where $\Delta_{p(x)}u={\rm{div}}\left(|\nabla u|^{p(x)-2}\nabla u \right)$ and the variable discontinuous exponent $p(x)$  is given by 
\begin{equation}\label{function-p(x)}
p(x)=
\left\{ \begin{array}{cl}
2& \hbox{ if } x \in D_1,
\\
p>2 & \hbox{ if } x \in D_2.
\end{array}
\right.
\end{equation}

  With regard to equations involving $p(x)$-Laplacian terms, with a general $p(x)$ (not necessarily discontinuous)  we refer the reader to the recent book \cite{DHHR} for background and an extensive review of recent results. In addition, problems that involve the $p(x)$-Laplacian with a discontinuous variable exponent, which is assumed to be constant in disjoint pieces of the domain $\Omega$, are recently used to model organic semiconductors (i.e., carbon-based materials conducting an electrical current). In these models $p(x)$ describes a jump function that characterizes Ohmic and non-Ohmic contacts of the device
material, see \cite{BGL1} and \cite{BGL2}. 
In fact, let us consider the Organic Light-Emitting Diodes (OLEDs) which are constituted by thin-film heterostructures made up by organic molecules or polymers. Each functional layer has its own current-voltage characteristics and hence, the current-flow equation is of $p(x)$-Laplacian type. Since the exponent $p(x)$ describes non-Ohmic behavior of materials,  it changes abruptly in passing from one to another. For example, in electrodes the parameter $p(x)$ is typically $2$ (Ohmic) while in organic materials $p(x)$ takes larger values, e.g. $p(x)=9$ (\cite{FKGBLLGS}).

This work is devoted to the study of this kind of operators with a power nonlinearity on the right hand side
that has a concave-convex nature with respect to the variable operator $\Delta_{p(x)}$. That is, convex (superlinear) for the Laplacian and concave (sublinear) for the $p$-Laplacian. Concretely, we look for existence and multiplicity of positive weak solutions for the following problem
\begin{equation} \label{alexis}
\left\{\begin{array}{ll}
- \Delta u = \lambda u^q , \qquad & \mbox{ in } D_1, \\[5pt]
- \Delta_p u = \lambda u^q , \qquad & \mbox{ in } D_2 ,\\[5pt]
\displaystyle \frac{\partial u }{\partial \eta} = |\nabla u|^{p-2}
\frac{\partial u }{\partial \eta},\qquad u|_{D_1} = u|_{D_2} , \qquad & \mbox{ on } \Gamma, \\[5pt]
u=0, \qquad & \mbox{ on } \partial \Omega,
\end{array} \right.
\end{equation}
 in the following function space
\[
\w=\left \{v\in W_0^{1,2}(\Omega) : \int_{D_2}|\nabla v|^p<\infty  \right \}.
\]
Here
\[
\lambda>0,\qquad 2<q+1<\,p ,
\]
and $\eta$ is the normal unit vector to $\Gamma$ pointing outwards $D_1$. This space $\w$ is a reflexive and separable Banach space equipped with the norm 
\begin{equation}\label{norm}
[v]_{\w}:=\|\, \nabla v\,\|_{L^2(D_1)}+\|\, \nabla v\,\|_{L^p(D_2)}\,
\end{equation}
(see Lemma \ref{space} for a detailed proof). We refer to the Preliminaries section in order to justify the definition of this convenient space.

  Observe that in \eqref{alexis} we have continuity of the solution, in the sense that the trace of $u$ on $\Gamma$
coincides coming from $D_1$ and coming from $D_2$, and also we have continuity of the associated fluxes across $\Gamma$.
 In addition, note that the exponent $q$ is a superlinear exponent (convex) for the problem in $D_1$ and a $p-$sublinear one (concave) for the problem in $D_2$. Therefore this problem has both a concave part and a convex one
 (but acting in different regions).

  It is fairly easy to see that problem \eqref{alexis} has a variational structure. Indeed, if we consider the functional $F:\w\to \RR$
\begin{equation}\label{functional}
F_\lambda(u) = \int_{D_1} \frac{|\nabla u|^2}{2} \, dx +
\int_{D_2} \frac{|\nabla u|^p}{p} \, dx - \lambda \int_{\Omega} \frac{|u|^{q+1}}{q+1} \, dx,
\end{equation}
as we will see in Lemma \ref{critical-points}, positive solutions of \eqref{alexis} are uniquely identified as being positive critical points for this functional.

From a pure mathematical perspective concave--convex problems have received some interest in the
literature in recent times, including several kinds of boundary conditions and generalizations to other 
operators such as the $p$--Laplacian or fully nonlinear uniformly elliptic operators. The subject goes back to
the pioneering works \cite{BEP}, \cite{GP}, \cite{GP1} and \cite{Lio}. However, \cite{ABC} is
regarded as a first detailed analysis of the main properties of such type of problems, especially
its bifurcation diagrams (see also \cite{Lio}, Section 1.1). We also quote \cite{AGP} and
\cite{GAPM} that deal with Dirichlet conditions and the $p$--Laplacian operator; \cite{CCP}, dedicated 
to fully nonlinear uniformly elliptic operators with Dirichlet boundary conditions; \cite{GPR}, dealing with flux--type nonlinear boundary
conditions and source nonlinearities and \cite{GRS} handling concave--convex terms of absorption nature. Of course,  this list is far from being complete and is only a sample of the previous research on the topic.

  In this framework we have the following results:

\begin{theorem}\label{Teo1} There exists $\lambda^*>0$ such that:
\begin{enumerate}
\item For $0<\lambda <\lambda^*$ there exists $w_\lambda$ a minimal positive solution. Moreover, 
this minimal solution, $w_\lambda$, is unique and increasing with respect to $\lambda$.
\\
\item For $\lambda > \lambda^*$ there is no positive solution.
\end{enumerate}
\end{theorem}
The proof is based on the method of sub and supersolution. For this, a comparison principle and
a maximum principle for this problem are needed. For the nonexistence of solutions for $\lambda$ large we use the fact that solutions to the parabolic problem $u_t=\Delta u +\lambda u^q$ in $D_1$, with large initial data, blow up in finite time. Theorem \ref{Teo1} is proved in Section 3.

Our next result shows that this problem has a second solution for almost every 
$0<\lambda < \lambda^*$ when $p$ is subcritical, in our case that is, $p<2^*$. Here
$2^*=\frac{2N}{N-2}$ if $N\geq 3$ and $2^*=\infty$ when $N=1,2$. Note that we also have that $q$ is subcritical
since $1<q<p-1 <2^*-1$.

\begin{theorem}\label{Teo2} Assume, in addition, $p<2^*$ and $D_2\subset \subset \Omega$. Then, there exists a second positive solution $ v_\lambda$ for almost every $0<\lambda < \lambda^*$. 
\end{theorem}



To prove the existence of a second solution we argue in two steps:
 First, using variational methods, we prove that \eqref{alexis} has a solution which is a local minimum of the corresponding energy functional (Theorem \ref{minimum}). This fact is subtle and we run into new difficulties. To be more precise, as the operator acts differently in $D_1$ and in $D_2$, we can only get regularity of solutions at locally Hölder 
 spaces (we refer the seminal paper \cite{Fusco}). Then, to show that there is a local minimum in 
 $\w$, we assume that $D_2\subset \subset \Omega$ in order to get $\mathcal{C}^1$ regularity close to $\partial \Omega$ 
 and then we show that there is a minimum in the stronger topology $\mathcal{C}^1(F_\delta)\cap \mathcal{C}(\overline \Omega)$ where $F_\delta$ is a small strip around the boundary of $\Omega$. Then, by using a delicate regularity argument, we relax the topology to $\w$. Here we use partially the ideas from \cite{ABC,BN,GAPM} adapting them to our setting with the introduction of a new original trick while using Stampacchia's approach in Proposition \ref{trick-Stampacchia} in order to obtain an $L^\infty-$bound. It is at this point where we use that $p<2^*$. Note that our space of solutions $\w$ is a subspace of $W_0^{1,2}(\Omega)$ that is larger than $W_0^{1,p}(\Omega)$.
 
Next, in order to prove the existence of a second positive solution, the crucial fact is to try to apply a Mountain Pass argument. The main difficulty here is to show that Palais-Smale sequences are bounded in $\w$. This question is at present far from being solved and an affirmative answer would allow to find a second solution for all $\lambda \in (0,\lambda^*)$ instead of for almost every $\lambda \in (0,\lambda^*)$. Let us discuss some difficulties: Initially, we point out that the usual trick combining $F_\lambda(u_n) \to c$ with $F_\lambda'(u_n) u_n = o(\|u_n\|)$ does not work here. In addition, we would like to comment that in previous references involving the search for critical points of Mountain Pass type for problems like
  \[
 \left\{
 \begin{array}{cc}
 -\Delta u = f(x,u), & \hbox{ in } \Omega,
 \\
 u=0,& \hbox{ on } \partial \Omega,
 \end{array}
 \right.
 \]
 it is usually assumed that 
 \[
 \exists \, \kappa >2 \, \hbox{ such that }\, \forall \, s\geq 0 \hbox{ and a.e. } x\in \Omega  \Rightarrow 0\leq \kappa F(x,s) \leq sf(x,s),
 \tag{AR}
 \]
 where $F(x,s)=\int_0^sf(x,t)dt$. This condition was originally introduced in \cite{AR} and it is called Ambrosetti-Rabinowitz type condition. Roughly speaking, the role of (AR) is to ensure that all Palais-Smale sequences  at the mountain pass level are bounded. Adapting this result to our variable operator $\Delta u \chi_{D_1}+ \Delta_p u \chi_{D_2}$ it is not difficult to prove that if $f(x,s)$ satisfies property (AR) for $\kappa >p$, then we have that Palais-Smale sequences are bounded (see Appendix). However, in our setting $f(x,s)=\lambda s^{q}$ and (AR)  is not satisfied for $\kappa >p$ because $q+1<p$. Moreover, even conditions weaker than (AR) present in the literature of elliptic equations ensuring  the existence of bounded Palais-Smale sequences are not applicable to our problem.
    To tackle this obstacle, we use some results from the classic works \cite{AR,Figue1,G-P,Jeanjean} again adapting them to our framework.  Mainly,  relying on a result by Jeanjean \cite{Jeanjean} which shows the existence a bounded
Palais-Smale sequence at mountain pass level for almost every $0<\lambda < \lambda^*$. We remark that once we have a bounded Palais-Smale sequence we are able to prove that there is a subsequence that converges strongly in $\w$.

Finally, we note that with the same ideas used here we can obtain similar results for the following problem
\begin{equation*}
\left\{\begin{array}{ll}
- \Delta u = \lambda u^{q_1} \chi_{D_1} + \lambda u^{q_2} \chi_{D_2} , \qquad & \mbox{ in } \Omega \\[5pt]
u=0, \qquad & \mbox{ on } \partial \Omega,
\end{array} \right.
\end{equation*}
with $q_1<1<q_2$. See \cite{GMRS} for similar results for the same problem with $\lambda u^{q(x)}$, with a
continuous exponent $q(x)$.

Also remark that when we take $D_1=D_2 = \Omega$, that is, for the problem
\begin{equation*}
\left\{\begin{array}{ll}
- \Delta u - \Delta_p u = \lambda u^{q}, \qquad & \mbox{ in } \Omega \\[5pt]
u=0, \qquad & \mbox{ on } \partial \Omega,
\end{array} \right.
\end{equation*}
with $1<q<p-1$ one has existence of a minimal positive solution for large $\lambda$, $\lambda > \tilde{\lambda}$ 
and nonexistence for small $\lambda$, $\lambda < \tilde{\lambda}$. This result (that can be obtained
just constructing adequate
sub and supersolution) has to be contrasted with ours for \eqref{alexis} where we have existence for small $\lambda$ and nonexistence for large $\lambda$.

\medskip

The rest of this paper is organized as follows: in the Preliminaries, Section \ref{sect-prelim}, we give some definitions and motivate the 
use of the space $\w$. In Section 3 we deal with the proof of Theorem \ref{Teo1}. Finally, in Section 4 we prove the existence of a second solution provided $p<2^*$. For completeness, in the Appendix we include a proof that
shows that Palais-Smale sequences are bounded when we assume (AR) with $\kappa>p$.

\section{Preliminaries} \label{sect-prelim}
In this section we motivate the 
use of the space $\w$ to define weak solutions for our problem and also we collect some results that will be used throughout this work.

   In order to justify the definition of space $\mathcal{W}(\Omega)$, let us give a briefly description about $W_0^{1,p(x)}$ spaces with $p(x)$ defined in \eqref{function-p(x)}. Following \cite{DHHR} we define the Banach space 
  \[
  L^{p(x)}(\Omega)=\left\{v: \Omega \to \RR \, \hbox{ mesurable }\,:\, \|v\|_{L^2(D_1)}+\|v\|_{L^p(D_2)} <\infty \right \}.
  \]
 equipped with the Luxemburg norm
 \[
 \|v\|_{L^{p(x)}(\Omega)}=\inf_{\tau >0}\left \{ \int_{D_1}\left(\frac{u}{\tau} \right)^2 +\int_{D_2}\left(\frac{u}{\tau} \right)^p \leq 1\right\}.
 \] 
 The space $ L^{p(x)}(\Omega)$ is a reflexive and separable Banach space.  Accordingly, we set the Sobolev space 
 \[
 W^{1,p(x)}(\Omega)=\left\{v: \Omega \to \RR \, \hbox{ mesurable }\,:v,\,|\nabla v| \in L^{p(x)}(\Omega)\,\right \}
 \]
and we have that $W^{1,p(x)}(\Omega)$ is a reflexive and separable Banach space with the norm
 \[
 \|v\|_{W^{1,p(x)}(\Omega)}=\|v\|_{L^{p(x)}(\Omega)}+\|\,\nabla v\,\|_{L^{p(x)}(\Omega)}.
 \]
Moreover, since $C^\infty(\overline \Omega)$ is dense in $W^{1,p(x)}(\Omega)$ (\cite[Theorem 2.4 and 2.7]{FWZ}). Then, $W_0^{1,p(x)}(\Omega)$ is well-defined as the closure of $C_c^\infty(\Omega)$ in $W^{1,p(x)}(\Omega)$ and it satisfies
\[
W_0^{1,p}(\Omega) \subset W_0^{1,p(x)}(\Omega) \subset W_0^{1,2}(\Omega).
\]

  However, we can not use Poincaré's inequality in $\W$ since, in general, it does not hold for discontinuous exponents, see \cite[Sec. 8.2]{DHHR}. Thus, we deal with a different Sobolev space that will be appropriate for our problem. Concretely, we define the Sobolev space $\mathcal{W}(\Omega)$
 \[
 \mathcal{W}(\Omega)=\left \{v\in W_0^{1,2}(\Omega) : \int_{D_2}|\nabla v|^p<\infty  \right \},
 \] 
 equipped with the following norm
 \[
 \|v\|_{ \mathcal{W}(\Omega)}=\|v\|_{W_0^{1,2}(\Omega)}+\|\, \nabla v \,\|_{L^p(D_2)}.
 \]
The space $\mathcal{W}(\Omega)$ is a separable and reflexive Banach space, since it is a closed subspace of $W_0^{1,2}(\Omega)$. The following result asserts that, by using Poincaré inequality, we can use the norm $[\, \cdot \,]_{\w}$ defined in \eqref{norm} which only depends on the gradient terms.
\begin{lemma}\label{space}
 $\left(\mathcal{W}(\Omega),[\, \cdot \,]_{\w}\right)$ is a reflexive and separable  Banach space.
\end{lemma}
\begin{proof}
 Since $\left(\w, \|\cdot \|_{\w}\right)$ is a reflexive and separable Banach space, it is sufficient to show that the norms $[\, \cdot \,]_{\w}$ and $\|\cdot \|_{\w}$ are equivalent. 
 For this purpose we use the fact that functions in the classical Sobolev space $W_0^{1,2}(\Omega)$ satisfies the Poincaré inequality and also that the continuous embedding of variable Lebesgue spaces to obtain for arbitrary $v\in \w$,
\begin{align}
\hspace{1.85cm} \nonumber  \|v \|_{\w} & =\|v\|_{W^{1,2}(\Omega)}+\|\,\nabla v\,\|_{L^p(D_2)}
\\
\nonumber  &\leq c_1 \|\, \nabla v \,\|_{L^2(D_1)}+c_1\|\, \nabla v \,\|_{L^2(D_2)}+\|\, \nabla v\,\|_{L^p(D_2)}
\\
\nonumber  &\leq c_1 \|\, \nabla v \,\|_{L^2(D_1)}+c_2\|\, \nabla v \,\|_{L^p(D_2)}
\\
\nonumber  &\leq c_3 \left(  \|\, \nabla v \,\|_{L^2(D_1)}+\|\, \nabla v \,\|_{L^p(D_2)}  \right).
\end{align} 
and
\begin{align}
\hspace{-18cm} \nonumber  \|v \|_{\w} & \geq  \|\, \nabla v \,\|_{L^2(\Omega)}+\|\, \nabla v \,\|_{L^p(D_2)}
\\
\nonumber  &\geq  \|\,\nabla v\,\|_{L^2(D_1)}+\|\,\nabla v\,\|_{L^p(D_2)}.
\end{align}
In these estimates, positive constants are denoted by $c_i,\, i\geq 1$.
 \end{proof}

\begin{remark} {\rm
It is worth pointing out that the $\|\cdot \|_{L^p(D_2)}$-norm is controlled by the $[\, \cdot \,]_{\w}$-norm 
(in particular, if $[u]_{\w}<\infty \Rightarrow \|u \|_{L^p(D_2)}<\infty$). Moreover, there exists $C>0$ such that $ \|u \|_{L^p(D_2)}\leq C \left( \|\,\nabla u\, \|_{L^p(D_2)}+ \|u \|_{L^2(D_2)}  \right)$. 
To see this fact, arguing by contradiction, suppose that for every $n \in \NN$ there exists $u_n$ such that
\begin{equation}\label{converges}
\|u_n\|_{L^p(D_2)}>n\left( \|\,\nabla u_n\, \|_{L^p(D_2)}+ \|u_n \|_{L^2(D_2)} \right)
\end{equation}
which is equivalent to write the above expression as
\[
1>n\left( \|\,\nabla v_n\, \|_{L^p(D_2)}+ \|v_n \|_{L^2(D_2)} \right).
\]
being $$v_n=\displaystyle\frac{u_n}{\|u_n\|_{L^p(D_2)}}.$$ Since $ \|\,\nabla v_n\, \|_{L^p(D_2)}<\frac{1}{n}$ and $\|v_n\|_{L^p(D_2)}=1$ it follows that the sequence $\{v_n\}$ is bounded in $W^{1,p}(D_2)$ and hence, up to a subsequence, $v_n$ converges weakly to $w\in W^{1,p}(D_2)$. Consequently, $v_n \to w$ in $L^r(D_2)$ for every $r\in [2,p^*)$. Taking $r=p$, and the fact $\|v_n\|_{L^p(D_2)}=1$ implies $\|w\|_{L^p(D_2)}=1$. However, taking $r=2$ from \eqref{converges} we have $\|u_n \|_{L^2(D_2)}<\frac{1}{n}$ and then we get that $\|w\|_{L^2(D_2)}=0$ 
leading to a contradiction.}
\end{remark}

 \begin{remark} {\rm Let $\w^\prime$ be the dual space of $\w$. We have that for every fixed $w\in \w$ the functional $\hat w:\w \to \RR$ defined as
 \[
 \hat w(v):=\int_{D_1}\nabla w \nabla v +\int_{D_2}|\nabla w|^{p-2}\nabla w \nabla v+\int_{\Omega}wv,\qquad v\in \w
 \]
belongs to $\w^\prime$.}
 \end{remark}

  Since we are considering positive solutions to the following $p(x)$-laplacian equation
 \begin{equation*}
\left\{\begin{array}{cl}
-\Delta_{p(x)}u=\lambda u^q, & \mbox{ in } \Omega,
\\
u=0, & \mbox{ on } \partial \Omega,
\end{array} \right.
\end{equation*}
with $p(x)$ defined in \eqref{function-p(x)}, a natural idea of what is a positive weak solution is a positive function 
that vanishes on $\partial \Omega$ (in an appropriate trace sense) and such that
\[
\int_{\Omega}|\nabla u|^{p(x)-2}\nabla u \nabla \varphi=\int_{D_1}\nabla u \nabla \varphi+\int_{D_2}|\nabla u|^{p-2}\nabla u \nabla \varphi=\lambda \int_{\Omega}u^{q}\,\varphi
\]
for all $\varphi \in \mathcal{C}_c^\infty (\Omega)$.

Hence, let us state the definition of weak positive solutions to our problem as follows:

\begin{definition}\label{weak-solution} Let $u\in \w$ be a positive function, it is said that $u$ is a weak positive solution of \eqref{alexis} if it satisfies 
\begin{equation}\label{derivate-functional.def}
\int_{D_1}\nabla u \nabla \varphi+\int_{D_2}|\nabla u|^{p-2}\nabla u \nabla \varphi=\lambda \int_{\Omega}u^{q}\,\varphi
\end{equation}
for all $\varphi \in \mathcal{C}_c^\infty(\Omega)$.
\end{definition}

Note that \eqref{derivate-functional.def} is formally equivalent to
the following conditions:
\[
\int_{D_1}\nabla u \nabla \varphi = \lambda \int_{D_1}u^{q}\,\varphi +\int_{\Gamma}\frac{\partial u}{\partial \eta}\,\varphi,
\]
\[
\int_{D_2}|\nabla u|^{p-2}\nabla u \nabla \varphi = \lambda \int_{D_2}u^{q}\,\varphi -\int_{\Gamma}|\nabla u|^{p-2}\frac{\partial u}{\partial \eta}\,\varphi,
\]
and
\[
\int_{\Gamma}\frac{\partial u}{\partial \eta}\,\varphi =\int_{\Gamma}|\nabla u|^{p-2}\frac{\partial u}{\partial \eta}\,\varphi.
\]

In the next lemma we prove that we can study critical points of functional \eqref{functional} instead of solutions of equation \eqref{alexis}.

\begin{lemma}\label{critical-points} Solutions of \eqref{alexis} are characterized by positive critical points of functional in \eqref{functional} 
\end{lemma}

\begin{proof} From Definition \ref{weak-solution}, weak solutions satisfy
\begin{equation*}
\int_{D_1}\nabla u \nabla \varphi+\int_{D_2}|\nabla u|^{p-2}\nabla u \nabla \varphi=\lambda \int_{\Omega}u^{q}\,\varphi
\end{equation*}
for all $\varphi \in \mathcal{C}_c^\infty(\Omega)$. Therefore, weak solutions are positive critical points of the functional \eqref{functional}. Conversely, if $u\in \w$ is a critical point, we obtain in particular that
\begin{equation*}
\int_{D_1}\nabla u \nabla \phi =\lambda \int_{D_1}|u|^{q-1}u\,\phi, \quad \forall \, \phi \in \mathcal{C}_c^\infty(D_1).
\end{equation*}
Thus, $u$ is a weak solution of the laplacian problem: $-\Delta u =\lambda |u|^{q-1}u$ in $D_1$. Hence, multiplying by test functions $\varphi \in \mathcal{C}_c^\infty(\Omega)$, integrating by parts and taking into account that $\Gamma=\partial D_1\cap \Omega$, we obtain
\begin{equation}\label{laplacian1}
\int_{D_1}\nabla u \nabla \varphi = \lambda \int_{D_1}|u|^{q-1}u\,\varphi +\int_{\Gamma}\frac{\partial u}{\partial \eta}\,\varphi,
\end{equation}
being $\eta$ the normal unit vector to $\Gamma$ pointing outwards $D_1$. Analogously, choosing  test functions belongs to $\mathcal{C}_c^\infty(D_2)$, we get that critical points are weak solutions to the $p$-laplacian problem: $-\Delta_p u =\lambda |u|^{q-1}u$ in $D_2$. The same arguments used above applied to this case give
\begin{equation}\label{plaplacian1}
\int_{D_2}|\nabla u|^{p-2}\nabla u \nabla \varphi = \lambda \int_{D_2}|u|^{q-1}u\,\varphi -\int_{\Gamma}|\nabla u|^{p-2}\frac{\partial u}{\partial \eta}\,\varphi.
\end{equation}
Finally, since equalities \eqref{laplacian1} and \eqref{plaplacian1}  hold together, the fact that $u$ is a critical point imply that $\int_{\Gamma}\frac{\partial u}{\partial \eta}\,\varphi =\int_{\Gamma}|\nabla u|^{p-2}\frac{\partial u}{\partial \eta}\,\varphi$. Therefore, it follows that positive critical points of functional $F_\lambda$ are weak solutions to our problem.
\end{proof}

Finally, let us introduce the concept of sub and supersolution.

  \begin{definition}
  \label{epe} By a subsolution (respectively, supersolution) to the problem \eqref{alexis} we mean a function $u\in \w$ that satisfies the following inequality:
\[
\int_{D_1}\nabla u \nabla \varphi +\int_{D_2}|\nabla u |^{p-2}\nabla u \nabla \varphi  \leq
(\geq) \, \lambda \int_{\Omega}|u|^{q-1}u\,\varphi,
\]
for every $0\leq \varphi\in C_c^\infty(\Omega)$.
\end{definition}
Note that a solution is just a function which is both a subsolution and a supersolution.


\section{Existence and Non-Existence of Solutions}  \label{sect-existence}
This section deals with existence and non existence of solutions. Initially, note that the functional $F$ does not have a global minimum (and therefore the direct method of calculus of variations is not applicable). Indeed, let $v$ be a function in $\w$ with compact support in $D_1$, then, since we have that $q>1$,
\begin{equation}\label{MP-geom}
F_{\lambda}(t v) = t^2 \int_{D_1} \frac{|\nabla v|^2}{2} \, dx - t^{q+1} \lambda \int_{D_1} \frac{|v|^{q+1}}{q+1} \, dx 
\to -\infty
\end{equation}
as $t\to \infty$.

Hence,  we use sub and supersolution techniques in order to get existence of solutions to problem \eqref{alexis}. Our first step is to prove existence, uniqueness and a comparison principle for the problem
  \begin{equation}\label{for-f}
\left\{\begin{array}{ll}
- \Delta u = f , \qquad & \mbox{ in } D_1, \\[5pt]
- \Delta_p u = f, \qquad & \mbox{ in } D_2 ,\\[5pt]
\displaystyle \frac{\partial u }{\partial \eta} = |\nabla u|^{p-2}
\frac{\partial u }{\partial \eta} \qquad u|_{D_1} = u|_{D_2} , \qquad & \mbox{ on } \Gamma, \\[5pt]
u=0, \qquad & \mbox{ on } \partial \Omega.
\end{array} \right.
\end{equation}
Here solutions, sub and supersolutions are understood as in Definitions \ref{weak-solution} and \ref{epe} 
with $\lambda u^q$ replaced by $f$.
  
\begin{proposition}\label{f}
For every $f\in L^{2}(\Omega)$, the problem \eqref{for-f} has a unique weak solution in $u\in \w$.
\end{proposition}

\begin{proof}
It is sufficient to prove that the functional 
\[
I(u) := \int_{D_1} \frac{|\nabla u|^2}{2} \, dx +
\int_{D_2} \frac{|\nabla u|^p}{p} \, dx -  \int_{\Omega} f\,u \, dx ,
\]
has a unique critical point in $\w$. First, observe that  is straightforward that this functional  is weakly lower semi continuous in $\w$. Moreover,  there exists $0<C=C(N,p,\|f\|_{L^2(\Omega)},|\Omega|)$ such that
\[
I(u)\geq C\left( \|\,\nabla u\,\|_{L^2(D_1)}^2-\|\,\nabla u\,\|_{L^2(D_1)}+\|\,\nabla u\,\|_{L^p(D_2)}^p -\|\,\nabla u\,\|_{L^p(D_2)} \right).
\]
Thus, the functional is coercive (i.e., $I(u)\to \infty$ as $[u]_{\w}\to \infty$)  and since $\w$ is a reflexive Banach space there exists $u^*\in \w$ such that 
\[
I(u^*)=\min \{I(u)\,: \, u\in \w \}.
\]

  The uniqueness is due to the strict convexity of $I$. Indeed, by using the inequality  $|\xi|^r\geq |\xi_0|^r+r|\xi_0|^{r-2}\xi_0(\xi-\xi_0),$ for $\xi,\xi_0 \in \RR^N$ and $r=2,p$ (which is strict if $\xi \neq \xi_0$) it follows that $I(w)>I(v)+I^\prime (v)(w-v)$ for $v\neq w \in \w$.
\end{proof}

\begin{proposition}\label{CP-f}
Let $u_1,u_2\in \w$ be sub and supersolution respectively of \eqref{for-f}. Then $u_1\leq u_2$ a.e. in $\Omega$.
\end{proposition}

\begin{proof}
From the definition of sub and supersolution we get,
for every test function $0\leq \varphi \in \mathcal{C}_c^\infty(\Omega)$,
\begin{align}
\label{eq1}\int_{D_1}\nabla u_1 \nabla \varphi +\int_{D_2}|\nabla u_1 |^{p-2}\nabla u_1 \nabla \varphi &\leq  \int_{\Omega}f\,\varphi,
\\
\int_{D_1}\nabla u_2 \nabla \varphi +\int_{D_2}|\nabla u_2 |^{p-2}\nabla u_2 \nabla \varphi &\geq  \int_{\Omega}f\,\varphi.
\label{eq2}
\end{align}

Note that since $\w \subset W_0^{1,2}(\Omega)=\overline{\mathcal{C}_c^\infty(\Omega)}^{W^{1,2}}$, by density we can choose test functions in $\w$. In this way, consider the test function $$\varphi =(u_1-u_2)^+:=\max \left\{u_1-u_2,0  \right\}$$ in the above inequalities and subtract \eqref{eq2} from \eqref{eq1} to obtain 
\begin{align*}
\int_{\{x\in D_1:u_1>u_2\}}|\nabla (u_1-u_2)|^2 &
\\
+\int_{\{x\in D_2:u_1>u_2\}}&\left(|\nabla u_1|^{p-2}\nabla u_1-|\nabla u_2|^{p-2}\nabla u_2 \right)(\nabla u_1 -\nabla u_2)\leq 0.
\end{align*}
Finally, taking into account the well-known inequality 
\begin{equation}\label{class-ineq}
\left(|\xi|^{r-2}\xi-|\xi_0|^{r-2}\xi_0  \right)(\xi-\xi_0)\geq c(r)|\xi-\xi_0|^r,\,\quad \xi,\xi_0\in \RR^N, 
\end{equation}
for $r=2,p$, we conclude that $(u_1-u_2)^+ \equiv 0$ finishing the proof.
\end{proof}

As a direct consequence, there exists $u\geq 0$ the unique weak solution of \eqref{for-f} for every $0\leq f\in L^2(\Omega)$. The next result shows that in fact the solution is strictly positive when $f$ is nontrivial.

\begin{proposition}\label{positive123}For every nontrivial $0\leq f\in L^2(\Omega)$, every supersolution of \eqref{for-f} is strictly positive in $\Omega$.
\end{proposition}

\begin{proof}
Let $u\geq 0$ in $\Omega$ be a supersolution (or a solution) to \eqref{for-f}. There is no 
loss of generality in assuming that $f_{|D_2}\neq 0$ (the argument when $f_{|D_1}\neq 0$
is completely analogous). 
Consider $0<v\in W_0^{1,p}(D_2)$ the solution to the problem
\begin{equation}\label{in-D2}
\left\{\begin{array}{ll}
- \Delta_p v = f , \qquad & \mbox{ in } D_2, \\[5pt]
v=0 , \qquad & \mbox{ on } \partial D_2 .
\end{array} \right.
\end{equation}
Since $u\geq 0$, it follows that $u\geq 0$ on $\Gamma$ and hence $u$ is a supersolution to \eqref{in-D2}. From 
the comparison principle we obtain that $u\geq v>0$ in $D_2$. Furthermore, if $u(x_0)=0$ for some $x_0\in \Gamma$, by Hopf's lemma we have, in addition, that $$\displaystyle \frac{\partial u(x_0)}{\partial \eta}=|\nabla u(x_0)|^{p-2}\displaystyle\frac{\partial  u(x_0)}{\partial \eta}<0$$ which means that $x_0$ is not a minimum of $u$ and this contradicts the fact that $u(x_0)=0$. Therefore, $u>0$ on $\Gamma$. Finally, to show the that $u$ is positive in the region $D_1$, consider $w\in W^{1,2}(D_1)$ the solution to the following problem
\begin{equation}\label{in-D1}
\left\{\begin{array}{ll}
- \Delta w = 0 , \qquad & \mbox{ in } D_1, \\[5pt]
w=u , \qquad & \mbox{ on } \partial D_1 .
\end{array} \right.
\end{equation}
Since $u>0$ on $\Gamma \subset \partial D_1$, the strong maximum principle applied in problem \eqref{in-D1} shows that $w>0$ in $D_1$. Taking into account that $u$ is a supersolution to problem \eqref{in-D1}, we conclude 
from the comparison principle that $u\geq w>0$ in $D_1$.
\end{proof}

\begin{corollary}\label{positive}
Let $u\in \w$ be a nonnegative solution to problem \eqref{alexis}. Then either $u(x)=0$ a.e. $x$ in $\Omega$ or $u(x)>0$ a.e. $x\in \Omega$.
\end{corollary}
The method of proof of Proposition \ref{positive123} can be applied to solutions that are nonnegative 
and nontrivial on the boundary. To be more precisely, we state the following proposition whose proof is almost the same as the previous one and is therefore omitted.
\begin{proposition}\label{strong-principle}
Let $0\leq f\in L^2(\Omega)$ (maybe trivial) and $u$ solution of \eqref{for-f} with boundary conditions  $0 \lvertneqq u$ on $\partial \Omega$. Then $u>0$ in $\Omega$.
\end{proposition}

Now, we are ready to prove one of the main goals of this section.

\begin{proposition}\label{existence} There exists a minimal bounded and positive solution of problem \eqref{alexis} for every $0< \lambda \leq \tilde \lambda$, being $\tilde \lambda$ sufficiently small.
\end{proposition}
\begin{proof}
First, we find a supersolution of \eqref{alexis} for $\lambda$ small.
 By Proposition \ref{f}, let $\overline u\in \w$ be the unique positive solution to the problem
\begin{equation*}
\left\{\begin{array}{ll}
- \Delta w = 1 , \qquad & \mbox{ in } D_1, \\[5pt]
- \Delta_p w = 1 , \qquad & \mbox{ in } D_2, \\[5pt]
\displaystyle \frac{\partial w }{\partial \eta} = |\nabla w|^{p-2}
\frac{\partial w }{\partial \eta} ,  \qquad w|_{D_1} = w|_{D_2} , \qquad & \mbox{ on } \Gamma, \\[5pt]
w=0, \qquad & \mbox{ on } \partial \Omega.
\end{array} \right.
\end{equation*}

Classical regularity for $p$-laplacian operators states that  there exist $C_1,C_2>0$ such that $\|\overline u\|_{L^\infty(D_1)}\leq C_1$ and $\|\overline u\|_{L^\infty(D_2)}\leq C_2$. Furthermore, setting $\tilde \lambda=\frac{1}{(C_1+C_1)^q} $, we get
\begin{align}
\nonumber \int_{D_1} \nabla \overline u\, \nabla \varphi+\int_{D_2}|\nabla \overline u|^{p-2}\,\nabla \overline u \,\nabla \varphi = \int_\Omega \varphi = \tilde \lambda \int_\Omega (C_1+C_2)^q\,\varphi \geq \lambda \int_\Omega \overline u ^q \,\varphi,
\end{align}
for all $\lambda \leq \tilde \lambda$ and $0\leq \varphi \in \mathcal{C}_c^\infty(\Omega)$. Therefore, $\overline u$ is a supersolution of \eqref{alexis} for $\lambda \leq \tilde \lambda$. Note that this argument shows the existence of a bounded supersolution only for
$\lambda$ small. 

  Next, to get a subsolution, take $v\in W_0^{1,p}(D_2)$ the positive solution to
\begin{equation} \label{alexis44}
\left\{\begin{array}{ll}
- \Delta_p v = \lambda v^q , \qquad & \mbox{ in } D_2, \\[5pt]
v=0 , \qquad & \mbox{ on } \partial D_2 .
\end{array} \right.
\end{equation}
Note that there is a unique $v$ for every $\lambda>0$ due to the fact that $q<p-1$. Then we define
\begin{equation}\label{sub}
\underline{u} (x) = 
\left\{
\begin{array}{ll}
v (x) \qquad & x \in  D_2, \\
0 \qquad & x \in \overline D_1.
\end{array}
\right.
\end{equation}
Clearly, $\underline u $ belongs to $\w$. Moreover, due to Hopf's Lemma \cite{Sakaguchi}, we get that  $|\nabla \underline u|^{p-2}\frac{\partial \underline u}{\partial \eta}<0$ on $\Gamma$ (recal that $\eta$  is the normal unit vector to $\Gamma$ pointing outwards $D_1$), then
\begin{align}
\nonumber \int_{D_2} |\nabla \underline u|^{p-2}\nabla \underline u \nabla \varphi =\lambda \int_{D_2}\underline u^{q}\,\varphi+\int_{\Gamma}|\nabla \underline u|^{p-2}\frac{\partial \underline u}{\partial \eta}\,\varphi \leq \lambda \int_{D_2}\underline u^{q}\,\varphi,
\end{align}
for every $\lambda >0$ and $0\leq \varphi \in \mathcal{C}_c^\infty(\Omega)$. Thus, $\underline u$ is the required subsolution of \eqref{alexis} without any restriction on $\lambda>0$. We stress that, thanks to Hopf's Lemma, the above inequality is strict for tests functions that verify $\varphi >0$ on $\Gamma$. Thus, $\underline u$ is not a solution.

Clearly, $0=\underline u (x)\leq \overline u(x)$ for $x\in \overline D_1$. In addition, since $\underline u,\, \overline u$ are a solution and a supersolution respectively of problem \eqref{alexis44} for $\lambda \leq \tilde \lambda$, it follows by the comparison principle for $p-$sublinear terms in $p$-laplacian operators that 
$\underline u \leq \overline u$ a.e. in $D_2$. Finally, since $\underline u=\overline u =0$ on $\partial \Omega$, we can state that
\[
\underline u \leq \overline u,  \qquad \hbox{ a.e. in } \overline \Omega.
\]

To conclude, we use the standard monotone iteration argument in order to find a solution
for our problem. For every $n\geq 1$ we define the recurrent sequence $\{w_n\}$ by
\begin{equation}\label{recurrent}
\left\{\begin{array}{ll}
- \Delta w_n = \lambda\,w_{n-1}^q , \qquad & \mbox{ in } D_1, \\[5pt]
- \Delta_p w_n = \lambda\, w_{n-1}^q, \qquad & \mbox{ in } D_2, \\[5pt]
\displaystyle \frac{\partial w_n }{\partial \eta} = |\nabla w_n|^{p-2}
\frac{\partial w_n }{\partial \eta} , \qquad w_n|_{D_1} = w_n|_{D_2} , \qquad & \mbox{ on } \Gamma ,\\[5pt]
w_n=0, \qquad & \mbox{ on } \partial \Omega,
\end{array} \right.
\end{equation}
where $w_0=\underline u\,$. The sequence $\{w_n\}$ is well defined by Proposition \ref{f}. 
Moreover, the sequence is increasing. To check this property it suffices to prove that $w_0 \leq w_1$
(and then proceed by induction). Indeed, taking into account that $w_0$ is a subsolution of  problem \eqref{recurrent} for $n=1$, we obtain by comparison principle Proposition \ref{CP-f} that $w_0\leq w_1$. Hence, by an inductive argument: $w_0\leq w_1\leq \cdots \leq w_n$, for all $n\geq 1$. By the fact that $\overline u$ is a supersolution of  problem \eqref{recurrent} for $n=1$, with a similar argument we prove that $w_n \leq \overline u$ for every $n\in \NN$. Since $\overline u \in L^\infty(\Omega)$, the sequence $\{w_n(x)\}$ is increasing and bounded by $\overline u(x)$ for a.e. $x\in \Omega$. Let $w_\lambda(x)$ be the limit almost everywhere in $\Omega$ (i.e., $w_\lambda(x):=\lim_{n\to \infty}w_n(x)$ a.e. $x\in \Omega$) which is bounded since $\overline u$ is bounded. We claim that $w_\lambda \in \w$. Indeed, since $w_n\in \w$ we can take it as a test function in equation \eqref{recurrent}
to obtain
\[
\int_{D_1}|\nabla w_n|^2+\int_{D_2}|\nabla w_n|^{p}=\lambda\,\int_{\Omega}w_{n-1}^qw_n\leq \lambda\int_{\Omega} \overline u^{\,q+1}\leq \lambda \|\overline u\|^{\, q+1}_{L^\infty(\Omega)}|\Omega|.
\]
That is, $\{w_n \}$ is uniformly bounded in the norm of $\w$ and since this space is reflexive, up to a subsequence, $w_n$ converges weakly to $\tilde w \in \w$. Furthermore, $w_n(x)\to \tilde w(x)$ a.e. $x\in \Omega$. Finally,  by the uniqueness of the limit $w_\lambda =\tilde w\in \w$ and we conclude the claim.

  To finish the proof, we verify that $w_\lambda$ is a weak solution of \eqref{alexis}. To this end, fix $\varphi \in \mathcal{C}_c^\infty(\Omega)$ and observe that from \eqref{recurrent} we get
\[
\int_{D_1}\nabla w_n \nabla \varphi+\int_{D_2}|\nabla w_n|^{p-2}\nabla w_n \nabla \varphi =\lambda \, \int_{\Omega}w_{n-1}^q\,\varphi.
\]
Now, let $n \to \infty$ to obtain
\[
\int_{D_1}\nabla w_\lambda \nabla \varphi+\int_{D_2}|\nabla w_\lambda|^{p-2}\nabla w_\lambda \nabla \varphi =\lambda \, \int_{\Omega}w_{\lambda}^q\,\varphi,
\]
as desired.  We note that $w_\lambda$ is positive by Corollary \ref{positive} and minimal by construction. In fact, let $\tilde w_\lambda$ be another solution of problem \eqref{alexis}, by a similar argument using the comparison principle and induction in $n$ we obtain $w_n \leq \tilde w_\lambda$ for all $n \in \NN$, thus $w_\lambda(x)=\lim_{n\to \infty}w_n(x)\leq \tilde w_\lambda(x)$ a.e. $x\in \Omega$. 
\end{proof}

Now we are ready to proceed with the proof of Theorem \ref{Teo1}.

\begin{proof}[Proof of Theorem \ref{Teo1}]
First, we observe that if there exists $\hat u \in \w$, a solution to problem \eqref{alexis} for some $\hat \lambda>0$, then there exists $w_\lambda$ a minimal solution for every $\lambda \in (0,\hat \lambda)$. Indeed, for a fixed $0<\lambda<\hat \lambda$, we take $\hat u$ as a supersolution and $\underline u$ from \eqref{sub} as a subsolution of problem \eqref{alexis}. Recall that we have showed existence of this subsolution for any value of $\lambda >0$. Arguing as in the proof of Proposition \ref{existence}, it holds that the sequence $\underline u <w_1\leq w_2\leq \cdots \leq w_n \leq \cdots \leq \hat  u$ is uniformly bounded in $\w$ and, by our previous argument, there exists $w_\lambda$, the minimal solution. In this way we set
\[
\lambda^*=\sup \left\{0\leq \lambda\,:\, \hbox{exists a solution to problem }\eqref{alexis} \right\}.
\]
By Propositon \ref{existence} it follows that $\lambda^*>0$. Thus, for every $0<\lambda <\lambda^*$ there exists $w_\lambda$ a minimal positive solution.

Next, in order to prove that $\lambda^*<\infty$,  we take again  $v\in W_0^{1,p}(D_2)$ the unique positive solution to \eqref{alexis44} and let us observe that
$$
v (x) = \lambda^{\gamma} \,v_1(x), \qquad \mbox{ in } D_2,
$$
with $\gamma=\frac{1}{p-1-q}>0$ and $v_1$ the unique solution to 
\begin{equation*}
\left\{\begin{array}{ll}
- \Delta_p v_1 =  (v_1)^q , \qquad & \mbox{ in } D_2, \\[5pt]
v_1=0 , \qquad & \mbox{ on } \partial D_2.
\end{array} \right.
\end{equation*}
Now, fix a ball $B\subset \subset D_2$.
Since $v_1 \geq c>0$ in $B$, it holds that
$$
v (x) \geq c  \lambda^{\gamma}, \qquad x \in B.
$$ 
That is, $v$ is uniformly large in $B$ for $\lambda$ large.

Now, let us consider $z$ the solution to
\begin{equation} \label{alexis.99.kk}
\left\{\begin{array}{ll}
- \Delta z = 0 , \qquad & \mbox{ in } D_1 ,\\[5pt]
- \Delta_p z = 0 , \qquad & \mbox{ in } D_2, \setminus B \\[5pt]
\displaystyle \frac{\partial z }{\partial \eta} = |\nabla z|^{p-2}
\frac{\partial z }{\partial \eta} , \qquad z|_{D_1} = z|_{D_2} ,\qquad & \mbox{ on } \Gamma, \\[5pt]
z=0, \qquad & \mbox{ on } \partial \Omega, \\[5pt]
z= c \lambda^{\gamma} , \qquad & \mbox{ on } \partial B.
\end{array} \right.
\end{equation}
Such solution can be obtained as the minimum from the following coercive functional
$$
H(u) = \int_{D_1} \frac{|\nabla u|^2}{2} \, dx +
\int_{D_2\setminus B} \frac{|\nabla u|^p}{p} \, dx  
$$
in the set $\mathcal{A}=\{u\in \tilde{\mathcal{W}}(\Omega\setminus B)\,:\,u_{|_{\partial B}}\equiv    c \lambda^{\gamma}  \}$ being $\tilde{\mathcal{W}}(\Omega\setminus B)$ the Banach space defined as
\[
\tilde{\mathcal{W}}(\Omega\setminus B)=\left\{u\in W^{1,2}(\Omega\setminus B)\cap W^{1,p}(D_2\setminus B)\,:\, u_{|_{\partial \Omega}}\equiv 0   \right\}.
\]
We note that such minimum is attained because $\mathcal{A}$ is a nonempty convex and weakly close subset of $\tilde{\mathcal{W}}(\Omega\setminus B) $.

Now fix a different ball $B_2 \subset \subset D_1$. We claim that $z$ is uniformly large in $B_2$ when $\lambda$ is large. Indeed, $z$ should be large on $\Gamma$ and therefore large in $B_2$.

In order to prove the  nonexistence of solutions to 
\eqref{alexis} for $\lambda$ large.
Assume, arguing by contradiction, that there is a solution $u$ for $\lambda $ large. By a comparison argument,
we have that 
$$
u \geq v, \qquad \mbox{in } D_2.
$$
Hence $u$ is a supersolution of problem \eqref{alexis.99.kk} in $\tilde{\mathcal{W}}(\Omega\setminus B)$ and due to Proposition \ref{strong-principle}  in the space $\tilde{\mathcal{W}}(\Omega\setminus B)$, it holds by comparison principle
$$
u \geq z \qquad \mbox{ in } B_2.
$$
This gives a contradiction, since the solution to the parabolic problem
\begin{equation} \label{alexis.99.kk.parabol}
\left\{\begin{array}{ll}
w_t - \Delta w = \lambda w^q , \qquad & \mbox{ in } B_2 \times (0,T) ,\\[5pt]
w=0, \qquad & \mbox{ on } \partial B_2  \times (0,T) , \\[5pt]
w_0= z , \qquad & \mbox{ in }  B_2,
\end{array} \right.
\end{equation}
blows up in finite time (due to the fact that $z$ is uniformly large in the ball $B_2$, see for instance \cite{Ball}) and 
also must satisfy 
$$
w(x,t) \leq u(x),
$$
since $u$ is a supersolution to the parabolic problem \eqref{alexis.99.kk.parabol}.

Finally, we note that if $\lambda_1 \leq \lambda_2<\lambda^*$, taking $w_{\lambda_2}$ as a supersolution of problem \eqref{alexis} for $\lambda=\lambda_1$ and arguing as the proof of Proposition \ref{existence} we obtain $w_{\lambda_1} \leq w_{\lambda_2}$. That is, the family of functions $\{w_{\lambda}\}_{0<\lambda<\lambda^*}$ is increasing with $\lambda$.
\end{proof}

\section{Multiplicity of solutions} \label{sect-multi}
In this section we show that problem \eqref{alexis} has at least two positive different solutions provided $p<2^*$ if $N\geq 3$ (with no restriction on $p$ for $N=1,2$) and $D_2\subset \subset \Omega$. Concretely, we prove that \eqref{alexis} has a first solution which corresponds to the global minimum of an appropriated functional and then a second solution is found by means of Mountain Pass theory.

Since our objective is to find positive solutions of our problem, we observe that they correspond to critical points of the following functional
\begin{equation*}
G_\lambda(u) = \int_{D_1} \frac{|\nabla u|^2}{2} \, dx +
\int_{D_2} \frac{|\nabla u|^p}{p} \, dx - \lambda \int_{\Omega} \frac{u_+^{q+1}}{q+1} \, dx,
\end{equation*}
where $u_+=\max \{ u,0   \}$. We will write it simply $G$ instead $G_\lambda$ when no confusion can arise. Of course, $F(u)=G(u)$ whenever $u\geq0$ and then, positive critical points of $G$ correspond to positive solutions of $\eqref{alexis}$.

  In general, for a $p(x)$ discontinuous, the $\mathcal{C}^1(\Omega)$-regularity of minimizers of $G$ are not satisfied, in fact, one can find some counter-examples in \cite{Zhikov}. However, as it mentioned in \cite[Theorem 9.15]{Filandeses} which refers to \cite{FZ}, for our class of discontinuous exponents one can arrive at locally Hölder continuity (see also \cite{Fusco}). Therefore, due to lack of $\mathcal{C}^1$-results in whole $\Omega$, we impose that $D_2\subset \subset \Omega$ in order to get regularity close to $\partial \Omega$.  Concretely, as we will see later, we need that local minimizers of functional $G$ belongs to $\mathcal{C}^1(F_\delta)\cap \mathcal{C}(\overline \Omega)$ where $F_\delta$ is a small strip around the boundary,
  \begin{equation}\label{franja}
  F_\delta = \{x\in \Omega \,:\,  \hbox{dist}(x,\partial \Omega)<\delta \}
  \end{equation}
being $\delta$ enough small to ensure that $F_{3\delta} \subset   D_1$ and $\partial F_\delta$ is smooth.

Following partially the ideas in \cite{ABC}, we begin by showing the next result.
\begin{lemma}\label{topology}
For every $\lambda\in (0,\lambda^*)$ there exists a local minimum of $G$ in the $\mathcal{C}(\overline \Omega)\cap \mathcal{C}^1(F_\delta)$-topology. 
\end{lemma}

\begin{proof}
Fixed $0<\lambda<\lambda^*$, we take $\lambda_1,\lambda_2>0$ such that $\lambda_1<\lambda<\lambda_2<\lambda^*$ and let us denote by $u_1$ and $u_2$ their respective minimal solutions for $\lambda_1$ and $\lambda_2$ obtained  in Theorem \ref{Teo1}.  Since the minimal solutions are increasing, we have $ u_1  \leq u_2$. Even more, since $\lambda_1<\lambda_2$ it follows  by the Strong Maximum Principle applied in each region $D_i,\, i=1,2$ (see for instance \cite{Dama,Guedda-Veron}) and the Hopf Maximum Principle that
\begin{equation*}
\begin{array}{cc}
u_1<u_2, &  \hbox{ in } \Omega, \\[10pt]
\displaystyle \frac{\partial u_2}{\partial \nu}<\frac{\partial u_1}{\partial \nu}<0, & \hbox{ on } \partial \Omega,
\end{array}
\end{equation*}
being $\nu$ the outer unit normal on $\partial \Omega$.

  Consider, 
  \begin{equation*}
  h(x,s)=
  \left\{
  \begin{array}{lc}
  u_2^q(x)\,, & s\geq u_2(x),\\[5pt]
  s^q\,, & u_1(x)<s<u_2(x), \\[5pt]
  u_1^q(x)\,, & s\leq u_1(x),
  \end{array}
  \right.
  \end{equation*}
and the truncated functional
\[
\tilde{G}(u)=\int_{D_1}\frac{|\nabla u(x)|^2}{2}+\int_{D_2}\frac{|\nabla u(x)|^p}{p}-\lambda \int_{\Omega}H(x,u)
\]
where 
$u \in \w$ and $H(x,u)=\int_0^uh(x,s)ds$. Clearly, $\tilde G$ is coercive and weakly lower semicontinuous (because $q<\frac{N+2}{N-2}$) . Hence, there exists its global minimum at some $\tilde u \in \w$ and for every $0\leq \varphi \in \mathcal{C}_c^\infty(\Omega)$ it holds
\begin{align*}
\int_{D_1}\nabla \tilde u(x) \nabla \varphi(x) +\int_{D_2}|\nabla \tilde u(x)|^{p-2}\nabla \tilde u(x) \nabla \varphi(x)  &= \lambda \int_{\Omega}h(x,\tilde u) \varphi(x)
\\
& > \lambda_1 \int_{\Omega}u_1^q(x)\varphi(x).
\end{align*}
That is, $\tilde u$ is a supersolution of \eqref{for-f} with $f=\lambda_1 u_1^q$ and since $u_1$ is a solution it follows by the comparison principle from Proposition \ref{CP-f} that $u_1\leq \tilde u$. We proceed analogously to obtain that $\tilde u \leq u_2$. Moreover, using again the Strong Maximum Principle and the Hopf Maximum Principle we obtain that 
\begin{align}\label{ineq1}
0<u_1<\tilde u<u_2, &\quad \hbox{ in } \Omega,
\end{align}
and
\begin{align}\label{ineq2}
\displaystyle \frac{\partial u_2}{\partial \nu}<\frac{\partial \tilde u}{\partial \nu}<\frac{\partial u_1}{\partial \nu}<0, & \quad \hbox{ on } \partial \Omega.
\end{align}
Next, we claim that $\tilde u \in \mathcal{C}(\overline \Omega)\cap \mathcal{C}^1(F_\delta)$. Indeed, let $K=\Omega \setminus F_{\delta/2}$ be a compact set. Since $\tilde u$ is a local minimizer and $u_1,\,u_2$ are bounded then $\tilde G$ is in the framework of  the work \cite{FZ}. It follows a higher integrability of the gradient of $\tilde u$ which implies locally Hölder continuity, hence $\tilde u \in \mathcal{C}^\alpha(K)$. Moreover, $\tilde u$ satisfies the equation
\begin{equation*}
\left \{
\begin{array}{lc}
-\Delta \tilde u = \lambda \tilde u^q, & \hbox{ in } F_\delta,
\\
\tilde u=0, & \hbox{ on } \partial \Omega,
\end{array}
\right.
\end{equation*}
and $\tilde u$ is continuous on $\partial F_\delta \cap \Omega$. Then, the well-known classical regularity for the laplacian operator (see \cite{GT}) implies that $\tilde u\in \mathcal{C}^1(F_\delta)\cap \mathcal{C}(\overline F_\delta)$ and the claim is proved.

  Finally, in virtue of inequalities \eqref{ineq1} and \eqref{ineq2}, there exists $\varepsilon>0$ sufficiently small such that $u_1<v<u_2$ in $\Omega$ for all $v\in B_{\varepsilon}(\tilde u)$ the ball of center $\tilde u$ and radius $\varepsilon$ in the topology of $\mathcal{C}(\overline \Omega)\cap \mathcal{C}^1(F_\delta)$. Therefore,
  \[
  G(v)=\tilde G(v)\geq \tilde G(\tilde u)=G(\tilde u), \qquad \hbox{ for all } v\in B_{\varepsilon}(\tilde u).
  \]
Equivalently, $\tilde u$ is a local minimum of $G$ in $\mathcal{C}(\overline \Omega)\cap \mathcal{C}^1(F_\delta)$-topology.
\end{proof}

\begin{remark}\label{regul}{\rm
Concerning the regularity of local minimizers of functional $\tilde G$ in the proof of above lemma, the same reasoning  applied to the functional $G$ states that local minimizers of $G$ also belong to $\mathcal{C}(\overline \Omega)\cap \mathcal{C}^1(F_\delta)$.
}
\end{remark}

Our first goal is to show that there exists a local minimum of $G$ in $\w$. In fact, we will prove that $\tilde u$, the local minimum in $\mathcal{C}(\overline \Omega)\cap \mathcal{C}^1(F_\delta)$-topology of the proof of Lemma \ref{topology}, is the desired local minimizer. To prove it, we argue by contradiction following closely the ideas of  \cite[Lemma 1]{Figue} (see also \cite{BN}). Thus, we suppose that there exists $\varepsilon_0>0$ such that 
\begin{equation}\label{v}
G(v_\varepsilon):=\min \left \{ G(u) \, :\,   u \in V_\varepsilon  (\tilde u)\right \}<G(\tilde u), \quad \hbox{ for all } \varepsilon<\varepsilon_0,
\end{equation}
where $V_\varepsilon (\tilde u)$ is the closed set
 \[
 V_\varepsilon  (\tilde u)=\left \{u\in \w \,:\, \int_{D_1}\frac{|\nabla(u-\tilde u)|^2}{2} + \int_{D_2}\frac{|\nabla(u-\tilde u)|^p}{p}  \leq \varepsilon  \right \}.
 \]

  Note that such minimum is attained as $G$ is weakly lower semicontinuous and $V_\varepsilon  (\tilde u)$ is weakly compact in the reflexive space $\w$. Moreover, $v_\varepsilon \to \tilde u$ as $\varepsilon \to 0$ in norm in $\w$.

  The strategy is to prove that $v_\varepsilon \to \tilde u$ in $\mathcal{C}(\overline \Omega)\cap \mathcal{C}^1(F_\delta)$-topology contradicting the fact that $\tilde u$ is a local minimum in $\mathcal{C}(\overline \Omega)\cap \mathcal{C}^1(F_\delta)$-topology by the above lemma.
  
    For that purpose, we note that the corresponding Euler equation for $v_\varepsilon$ contains a nonpositive  Lagrange multiplier $\mu_\varepsilon\leq 0$. Namely,  $v_\varepsilon$ must be satisfy the following:
    \begin{align}\label{eq_v}
  \nonumber  \int_{D_1}\nabla u \nabla \varphi &+\int_{D_2}|\nabla u|^{p-2}\nabla u \nabla \varphi -\int_{\Omega}g(u)\varphi
    \\
 &\qquad =\mu_\varepsilon  \left[\int_{D_1}\nabla(u-\tilde u)\nabla \varphi + \int_{D_2}|\nabla(u-\tilde u)|^{p-2}\nabla(u-\tilde u)\nabla \varphi \right],
    \end{align}
    for all $\varphi \in \w$, being $g(u)=\lambda u_+^q$.
    
    Our first step is to prove that $v_\varepsilon$ are uniformly $L^\infty$-bounded by a constant independent of $\varepsilon$.
\begin{lemma}\label{bounded}
Given $0 \leq \varepsilon <\varepsilon_0<1$, there exists $M>0$ such that $v_\varepsilon$ defined by \eqref{v} satisfies $$\|v_\varepsilon \|_{L^\infty(\Omega)}\leq M,$$ for all $\varepsilon \in [0,\varepsilon_0)$.
\end{lemma}

\begin{proof} We adapt  the techniques applied in \cite{GAPM} by using the classical lemma due to Stampacchia \cite{Stamp}. First,  since 
\[
\int_{D_1}|\nabla \tilde u |^{p-2}\nabla \tilde u \nabla \phi + \int_{D_2}\nabla \tilde u \phi =\lambda \int_{\Omega}\tilde u^q\phi, \quad \forall \, \phi \in \mathcal{C}(\overline \Omega)\cap \mathcal{C}^1(F_\delta),
\]
and a density argument, the above equality holds for test functions belonging to $\w$.
Hence, we write equation \eqref{eq_v}, which satisfies $v_\varepsilon$, as follows
 \begin{align*}
  \int_{D_1}\nabla (u-\tilde u) \nabla \varphi &+\int_{D_2}(|\nabla u|^{p-2}\nabla u -|\nabla \tilde u|^{p-2}\nabla \tilde u)\nabla \varphi -\int_{\Omega}(g(u)-g(\tilde u))\varphi
    \\
 &\qquad =\mu_\varepsilon  \left[\int_{D_1}\nabla(u-\tilde u)\nabla \varphi + \int_{D_2}|\nabla(u-\tilde u)|^{p-2}\nabla(u-\tilde u)\nabla \varphi \right],
    \end{align*}
 for all $\varphi \in \w$.   
We consider now for every $k\in \RR^+$ the function $T_k:\RR \to \RR$ given by
\begin{equation*}
T_k(s)=
\left \{
\begin{array}{ll}
s+k, &\quad s\leq -k,
\\
0, & \quad -k<s\leq k,
\\
s-k, &\quad s>k.
\end{array}
\right.
\end{equation*}
Thus, taking $$\varphi =T_k(u-\tilde u)$$ as test function in the previous equation we get
 \begin{align*}
  \int_{D_1\cap \Omega_k}\nabla (u-\tilde u) \nabla T_k(u-\tilde u) +\int_{D_2\cap \Omega_k}(|\nabla u|^{p-2}\nabla u -|\nabla \tilde u|^{p-2}\nabla \tilde u)\nabla T_k(u-\tilde u)
    \\
=\int_{\Omega}(g(u)-g(\tilde u)) T_k(u-\tilde u)+\mu_\varepsilon  \left[\int_{D_1\cap \Omega_k}|\nabla(u-\tilde u)|^2 + \int_{D_2 \cap \Omega_k}|\nabla(u-\tilde u)|^{p} \right],
    \end{align*}
where $\Omega_k \equiv \left \{x\in \Omega \, : \, |u(x)-\tilde u(x)|>k    \right \}$.

  Hence, dropping the negative term $$\mu_\varepsilon  \left[\int_{D_1}|\nabla(u-\tilde u)|^2 + \int_{D_2}|\nabla(u-\tilde u)|^{p} \right]$$ and using the inequality \eqref{class-ineq}, we arrive to 
\begin{equation}\label{paso1}
\begin{array}{l}
\displaystyle
 \int_{D_1\cap \Omega_k} |\nabla T_k(u-\tilde u)|^2+c(p)\int_{D_2\cap \Omega_k} |\nabla T_k(u-\tilde u)|^p  
\\[10pt]
\displaystyle \qquad \leq \int_{\Omega}(g(u)-g(\tilde u)) T_k(u-\tilde u).
\end{array}
\end{equation}
We can also assume that $\|u-\tilde u\|_{L^r(\Omega)}\leq R$ independent of $\varepsilon$. Note that due $u \in V_\varepsilon(\tilde u)$ then $r$ is at least equal to $2^*$. Therefore, since $|T_k(s)|\leq |s|$ and applying Hölder inequality for this $r\geq 2^*$, the right hand side can be estimated as follows
\begin{equation}\label{paso2}
\begin{array}{l}
\displaystyle \int_{\Omega}(g(u)-g(\tilde u)) T_k(u-\tilde u)\leq \int_{\Omega_k}|g(u)-g(\tilde u)||T_k(u-\tilde u)|
\\[10pt] \qquad \displaystyle
 \leq \lambda \int_{\Omega_k}(|u|^q+|\tilde u|^q)|T_k(u-\tilde u)|
\\[10pt] \qquad \displaystyle \leq \lambda \left( \int_{\Omega_k}(|u|^q+|\tilde u|^q)^{\frac{r}{q}}\right)^{\frac{q}{r}} \left(  \int_{\Omega_k}|T_k(u-\tilde u)|^{2^*}   \right)^{\frac{1}{2^*}} \left|\Omega_k   \right|^{1-\frac{q}{r}-\frac{1}{2^*}}
\\[10pt] \qquad \displaystyle
 \leq C_1 \left(  \int_{\Omega_k}|T_k(u-\tilde u)|^{2^*}   \right)^{\frac{1}{2^*}} \left|\Omega_k   \right|^{1-\frac{q}{r}-\frac{1}{2^*}},
\end{array}
\end{equation}
for some positive constant $C_1(\lambda, q, N, R,\|\tilde u\|_{L^r(\Omega)})$. For the reader's convenience, we will explain the last inequality in more detail, we have
$$
\begin{array}{l}
\displaystyle \lambda \left( \int_{\Omega_k}(|u|^q+|\tilde u|^q)^{\frac{r}{q}}\right)^{\frac{q}{r}}\leq c_1(\lambda) \left( \int_\Omega |u|^{r}  +\int_\Omega |\tilde u|^{r}  \right)^{\frac{q}{r}}
\\[10pt] \qquad \displaystyle
\leq c_2(\lambda, q, N, \|u\|_{L^r(\Omega)},\|\tilde u\|_{L^r(\Omega)})
\\[10pt] \qquad \displaystyle
\leq c_3(\lambda, q, N, R, \|\tilde u\|_{L^r(\Omega)}).
\end{array}
$$

  Replacing inequality \eqref{paso2} in \eqref{paso1} we have that
\begin{align}\label{paso12}
\int_{D_1\cap \Omega_k} |\nabla T_k(u-\tilde u)|^2+c(p)\int_{D_2\cap \Omega_k} |\nabla T_k(u-\tilde u)|^p
\\
\nonumber \leq C_1 \left(  \int_{\Omega_k}|T_k(u-\tilde u)|^{2^*}   \right)^{\frac{1}{2^*}} \left|\Omega_k   \right|^{1-\frac{q}{r}-\frac{1}{2^*}}\,.
\end{align}
Concerning to the left hand side, we use the inequality
\[
a+b^{\,c}\geq 2^{-c}(a+b)^c, \qquad 0\leq a,b \leq 1 \leq c,
\]
to obtain
\begin{equation} \label{lado-derecho}
\begin{array}{l}
\displaystyle \int_{D_1\cap \Omega_k} |\nabla T_k(u-\tilde u)|^2+c(p)\int_{D_2\cap \Omega_k} |\nabla T_k(u-\tilde u)|^p 
\\[10pt] \qquad \displaystyle
\geq C_2\left( \int_{D_1\cap \Omega_k} |\nabla T_k(u-\tilde u)|^2 +\left( \int_{D_2\cap \Omega_k} |\nabla T_k(u-\tilde u)|^2\right) ^{\frac{p}{2}}\right)
\\[10pt] \qquad \displaystyle
\geq C_3 \left(  \int_{\Omega_k} |\nabla T_k(u-\tilde u)|^2  \right)^{\frac{p}{2}}
\\[10pt] \qquad \displaystyle
\geq C_4 \left(\int_{\Omega_k}  |T_k(u-\tilde u)|^{2^*} \right)^{\frac{p}{2^*}}.
\end{array}
\end{equation}
Going back to \eqref{paso12}, we get
\begin{equation}\label{paso3}
\left(\int_{\Omega_k}  |T_k(u-\tilde u)|^{2^*} \right)^{\frac{p-1}{2^*}} \leq C_5\,\left|\Omega_k   \right|^{1-\frac{q}{r}-\frac{1}{2^*}}\,.
\end{equation}

On the other hand, it is easy to check that $h-k \leq |T_k(s)|$, for $s\geq h \geq k$. Therefore, $h-k \leq |T_k(u-\tilde u)|$, for $x\in \Omega_h$ and $h\geq k$. Hence, we obtain the inequality
\begin{align}\label{guido}
|\Omega_h|(h-k)^{2^*}\leq \int_{\Omega_h}|T_k(u-\tilde u)|^{2^*}\leq \int_{\Omega_k}|T_k(u-\tilde u)|^{2^*}
\end{align}
and combining with \eqref{paso3} we have that
\[
|\Omega_h|\leq \frac{C_6}{(h-k)^{2^{^*}}}|\Omega_k|^{\beta},\qquad \mbox{ for } h>k.
\]
being $\beta=\left(1-\frac{q}{r}-\frac{1}{2^*}\right)\frac{2^*}{p-1}$. Therefore we can apply Stampacchia Lemma \cite{Stamp}, to deduce that
\begin{enumerate}
\item[(i)] if $u-\tilde u \in L^r(\Omega)$ with  $r>\displaystyle \frac{2^*q}{2^*-p}$, then $u-\tilde u \in L^\infty(\Omega)$ and $$\|u-\tilde u\|_{L^\infty(\Omega)}\leq c\,C_6^{1/2^*},$$ for some specific $c>0$,
\item[(ii)] if $u-\tilde u \in L^r(\Omega)$ with $r=\displaystyle \frac{2^*q}{2^*-p}$, then $u-\tilde u \in L^s(\Omega)$ for $s\in [1,\infty)$,
\\
\item[(iii)] if $u-\tilde u \in L^r(\Omega)$ with $r<\displaystyle \frac{2^*q}{2^*-p}$, then $u-\tilde u \in L^s(\Omega)$ for $s=\displaystyle \frac{2^*}{1-\beta}-\rho$ and $\rho>0$ arbitrary small.
\end{enumerate}
Since $u\in L^{2^*}(\Omega)$ we can argue as above for $r=2^*$. Thus, if $\,2^*> \frac{2^*q}{2^*-p}$ we conclude by item (i) that $u-\tilde u \in L^\infty(\Omega)$ and, in virtue of the regularity of $\tilde u$, we get that $\|u\|_{L^\infty(\Omega)}\leq M$. In the case $2^* = \frac{2^*q}{2^*-p}$ we use item (ii) to choose $s> \frac{2^*q}{2^*-p}$  and after repeating the argument we lie under the conditions of item (i) and conclude again the desired bound. Finally, in the case $2^*< \frac{2^*q}{2^*-p}$, by using item (iii) we can take
\[
r_1=\frac{2^*(p-1)}{p-2^*+q}-\rho_1>2^*.
\]
As before, if $r_1\geq \frac{2^*q}{2^*-p}$ we conclude easily. In other cases we take
\[
r_2=\frac{2^*(p-1)r_1}{(p-2^*)r_1+2^*q}-\rho_2.
\]
We claim that arguing by iteration, there exists $k_0\in \NN$ such that $r_k > \frac{2^*q}{2^*-p}$ for $k\geq k_0$, i.e, we can conclude after a  finite number of steps. Indeed, in other cases, we have that the sequence $\{r_k\}$ is bounded and it satisfies the recurrence 
\begin{equation}\label{recurrence}
\left\{
\begin{array}{l}
r_{k+1}=\displaystyle \frac{2^*(p-1)r_k}{(p-2^*)r_k+2^*q}-\rho_{k+1},
\\
r_0=2^*.
\end{array}
\right.
\end{equation}
Where $\rho_{k+1}\to 0$. Moreover, it is easy to check that the sequence is increasing and therefore it is convergent and the limit $r_\infty$ satisfies 
\[
r_\infty=\displaystyle \frac{2^*(p-1)r_\infty}{(p-2^*)r_\infty+2^*q},
\]
namely, $$r_\infty=\displaystyle \frac{2^*(p-1-q)}{p-2^*}<0,$$ which is a contradiction, proving the claim.
Note that here we use the condition $p<2^*$.
\end{proof}

\begin{remark}{\rm{
Note that the hypothesis $p<2^*$ is necessary in order  to apply Stampacchia's idea in the proof of the 
previous lemma.
}}
\end{remark}

\begin{proposition}\label{trick-Stampacchia}
Let $v_\varepsilon$ defined in \eqref{v}. Then $v_\varepsilon \to \tilde u$ in $\mathcal{C}(\overline \Omega)\cap \mathcal{C}^1(F_\delta)$-topology for $\delta>0$ sufficiently small.
\end{proposition}
\begin{proof}
Due to the construction of $F_\delta$ in \eqref{franja},  we have that $v_\varepsilon$ satisfies 
\begin{equation*}
\left\{
\begin{array}{lc}
-(1-\mu_\varepsilon)\Delta v_\varepsilon=\lambda v_\varepsilon^{\,q}, & \hbox{ in }F_{2\delta},
\\
v_\varepsilon =0, & \hbox{ on } \partial \Omega.
\end{array}
\right.
\end{equation*}
Moreover, by using Lemma \ref{bounded} it follows that $v_\varepsilon$ is bounded on $\partial F_{2\delta} \cap \Omega$. Then by interior regularity, one may bootstrap the bound $\|v_\varepsilon\|_{W^{1,2}(F_{\delta})}\leq M$ to arrive to $\|v_\varepsilon \|_{\mathcal{C}^{1,\alpha}(F_\delta)}\leq M$ independent of $\varepsilon$. Thus, since $v_\varepsilon \to \tilde u$ in $\w$ it follows by Arzelà-Ascoli  that $v_\varepsilon \to \tilde u$ in $\mathcal{C}^1(F_\delta)$. This concludes the first part of the proof.

  In order to prove that $v_\varepsilon \to \tilde u$ uniformly in $\mathcal{C}(\overline \Omega)$ we adapt part of the method of Stampacchia used in the proof of Lemma \ref{bounded} to get an estimate. Concrentely, let $\kappa \in \NN$ such that $r_{\kappa}$, the $\kappa$-term of the sequence \eqref{recurrence}, satisfies $r_\kappa > \frac{2^*q}{2^*-p}$. We adapted \eqref{paso2} replacing by $r_\kappa$ in the following form
$$
\begin{array}{l}
\displaystyle  
\int_{\Omega}(g(v_\varepsilon)-g(\tilde u)) T_k(v_\varepsilon-\tilde u)\leq \lambda \int_{\Omega_k}(|v_\varepsilon|^q+|\tilde u|^q)|T_k(v_\varepsilon-\tilde u)|
\\[10pt] \qquad \displaystyle
\leq \lambda \left( \int_{\Omega_k}(|v_\varepsilon|^q+|\tilde u|^q)^{\frac{r_\kappa}{q}}\right)^{\frac{q}{r_\kappa}} \left(  \int_{\Omega_k}|T_k(v_\varepsilon-\tilde u)|^{2^*}   \right)^{\frac{1}{2^*}} \left|\Omega_k   \right|^{1-\frac{q}{r_\kappa}-\frac{1}{2^*}}
\\[10pt] \qquad \displaystyle
\leq C \left(  \int_{\Omega_k}|T_k(v_\varepsilon-\tilde u)|^{2^*}   \right)^{\frac{1}{2^*}} \left|\Omega_k   \right|^{1-\frac{q}{r_\kappa}-\frac{1}{2^*}},
\end{array}
$$  
here $C=C(\lambda, q, \kappa, N, \|\tilde u\|_{L^{r_\kappa}(\Omega)})$. 
Let us consider $0<\tau <1/2^*$ sufficiently small, that we will specify later, and we write the last expression as follows
$$
\begin{array}{l}
\displaystyle
\int_{\Omega}(g(v_\varepsilon)-g(\tilde u)) T_k(v_\varepsilon-\tilde u)
\\[10pt] \qquad \displaystyle
\leq  C \left(  \int_{\Omega_k}|T_k(v_\varepsilon-\tilde u)|^{2^*}   \right)^\tau \left(  \int_{\Omega_k}|T_k(v_\varepsilon-\tilde u)|^{2^*}   \right)^{\frac{1}{2^*}-\tau} \left|\Omega_k   \right|^{1-\frac{q}{r_\kappa}-\frac{1}{2^*}}
\\[10pt] \qquad \displaystyle
\leq C \left(  \int_{\Omega}|v_\varepsilon-\tilde u|^{2^*}   \right)^\tau \left(  \int_{\Omega_k}|T_k(v_\varepsilon-\tilde u)|^{2^*}   \right)^{\frac{1}{2^*}-\tau} \left|\Omega_k   \right|^{1-\frac{q}{r_\kappa}-\frac{1}{2^*}}.
\end{array}
$$
Therefore, using this inequality in \eqref{paso1} and having in mind \eqref{lado-derecho}, it holds that
\[
\left( \int_{\Omega_k}|T_k(v_\varepsilon -\tilde u|^{2^*}     \right)^{\frac{p-1}{2^*}+\tau}\leq C\, \theta(\varepsilon)\, \left|\Omega_k   \right|^{1-\frac{q}{r_\kappa}-\frac{1}{2^*}},
\]
here $\theta(\varepsilon)=\left(  \int_{\Omega}|v_\varepsilon-\tilde u|^{2^*}   \right)^\tau$ (note that $\theta(\varepsilon) \to 0$ since $v_\varepsilon \to \tilde u$ in $\w$ ). Thus, by using inequality \eqref{guido}, we get
\[
|\Omega_h|\leq \frac{\tilde C \hat \theta(\varepsilon)}{(h-k)^{2^{^*}}}|\Omega_k|^{\hat\beta},\qquad h>k.
\]
Where $\hat \theta (\varepsilon)=\theta(\varepsilon)^{\frac{2^*}{p-1+\tau 2^*}}$ and
\[
\hat \beta =\frac{1-\frac{q}{r_\kappa}-\frac{1}{2^*}}{\frac{p-1}{2^*}+\tau}.
\]
Then, choosing $\tau$ such that $\hat \beta >1$ (note that it is possible due to the choice of $r_\kappa$) it  is straightforward by item (i) from Stampacchia Lemma that
\[
\|v_\varepsilon - \tilde u\|_{L^\infty(\Omega)}\leq c\, \hat \theta(\varepsilon)^{\frac{1}{2^*}}\to 0, \quad \hbox{ as } \varepsilon \to 0,
\]
which completes the proof.
\end{proof}

Summarizing, we have proved the following result:
\begin{theorem}\label{minimum}
For every $\lambda \in (0,\lambda^*)$, there exists, $\tilde u_\lambda$, a positive local minimum of $G_\lambda$ in $\w$.
\end{theorem}

The last goal is to obtain a second positive solution of problem \eqref{alexis}. Taking into account \eqref{MP-geom}, one may expect that  $G_\lambda$ possesses a mountain-pass geometry and, by using
results by Ghoussoub-Preiss (\cite{G-P}) and Jeanjean (\cite{Jeanjean}) in the spirit of the celebrated Mountain Pass theorem due to Ambrosetti and Rabinowitz (\cite{AR}), to find a critical point different from the minimum. To make sure that this critical point is nontrivial we consider, for every fixed $\lambda\in (0,\lambda^*)$, the  truncated functional $\widehat{G}_\lambda:\w \to \RR$ as follows:
\begin{equation}\label{hatG}
\widehat{G}_\lambda(u)=\int_{D_1}\frac{|\nabla u(x)|^2}{2}+\int_{D_2}\frac{|\nabla u(x)|^p}{p}-\lambda \int_{\Omega}\widehat{H}(x,u),
\end{equation}
as usual $\widehat{H}(x,s)=\int_0^s \widehat h(x,t)dt$, being in this case
 \begin{equation*}
  \widehat{h}(x,t)=
  \left\{
  \begin{array}{lc}
  t^q\,, & t>u_1(x) , \\[5pt]
  u_1^q(x)\,, & t\leq u_1(x),
  \end{array}
  \right.
  \end{equation*}
and  by $0<u_1$ we denote the minimal solution for a fixed $\lambda_1 \in (0,\lambda)$ which is obtained  in Theorem \ref{Teo1}. We point out that, $\widehat u_\lambda$, critical point of $\G$ corresponds to a supersolution of problem \eqref{for-f} with $f=\lambda_1u_1^q$. Hence, by Proposition \ref{CP-f}, it follows that $\widehat u_\lambda \geq u_1$. Moreover, if $\lambda > \lambda_1$ we obtain $\widehat u_\lambda > u_1$ and then it is also a critical point of $G_\lambda$.

In order to use the Mountain Pass theorem, as usual, a preliminary step is 
to show the existence of a bounded Palais-Smale sequence at the mountain pass 
level and then prove that it posses a convergent subsequence. We recall that a Palais-Smale
sequence for the functional $\G$ at level $c(\lambda)\in \RR$ is a sequence $\{u_n\}\subset \w$ verifying $\lim_n\G(u_n)=c(\lambda) $ and $\lim_n\G^{\,\prime}(u_n)=0$ in $\w^{\,\prime}$. We start by showing  
that bounded Palais-Smale sequences have a subsequence converging strongly in $\w$.
Note that we have to assume that the sequence is bounded, since it is not clear how to obtain boundedness in $\w$
using only that $\lim_n\G(u_n)=c(\lambda) $ and $\lim_n\G^{\,\prime}(u_n)=0$. 
This difficulty (showing that Palais-Smale sequences are bounded) forces us to use 
Jeanjean's ideas (\cite{Jeanjean}) and hence obtain existence of a second solution for almost every
$\lambda \in (0, \lambda^*)$.

\begin{lemma}\label{bounded-PS}
Let $\{u_n\}\subset \w$ be a sequence satisfying
\begin{enumerate}
    \item[(i)] \label{1} $\{u_n\}$ bounded in  $\w$, 
    
    \
    
    \item[(ii)] \label{2} $\G(u_n)$ bounded,
    
    \
    
    \item[(iii)] \label{3} $\G^{\, \prime}(u_n)\to 0$ in $\mathcal{W}^{\, \prime}(\Omega)$.
\end{enumerate}
Then, $\{u_n \}$ has a convergent subsequence in $\w$.
\end{lemma}

 \begin{proof}
{\emph{(i)}} there exists a subsequence $\{u_{n_k}\}$ and $u \in \mathcal{W}(\Omega)$, such that $u_{n_k} \rightharpoonup u$ in $\mathcal{W}(\Omega)$ and, by the embedding $\mathcal{W}(\Omega)\subset W_0^{1,2}(\Omega) \subset L^{r}(\Omega), \, \forall r\in[1,2^*) $, it holds $u_{n_k} \to u$ strongly in $L^{r}(\Omega) $.

  Let now $\varepsilon_{n_k}=\|\G^{\, \prime}(u_{n_k})\|_{\mathcal{W}^{\, \prime}(\Omega)}$. By {\emph{(iii)}} it holds $\varepsilon_{n_k} \to 0$. Furthermore
\begin{equation}\label{diff}
\left| \G^{\, \prime}(u_{n_k})(v)  \right|\leq \varepsilon_{n_k}[v]_{\w}, \qquad \forall v\in \mathcal{W}(\Omega), \, k\in \NN.
\end{equation}
  Choosing  $v=u_{n_k}-u$ in \eqref{diff} and taking into account that
  $$
  \int_{\Omega}\widehat{H}(x,u_{n_k}(x))(u_{n_k}-u)(x)\to 0
  $$
  (because $u_{n_k} \to u$ strongly in $L^{q+1}(\Omega) $, since $q+1<2^*$), we have from \eqref{diff} the following inequality
  \begin{equation*}\label{cond123}
  \int_{D_1}\nabla u_{n_k}\nabla (u_{n_k}-u)+\int_{D_2}|\nabla u_{n_k}|^{p-2}\nabla u_{n_k} \nabla (u_{n_k}-u)\leq \varepsilon_{n_k}[u_{n_k}-u]_{\w}.
  \end{equation*}
  And, since $\{u_n\}$ is bounded in norm $[\, \cdot \,]_{\w}$, it follows that
  \begin{equation}\label{cond456}
  \int_{D_1}\nabla u_{n_k}\nabla (u_{n_k}-u)+\int_{D_2}|\nabla u_{n_k}|^{p-2}\nabla u_{n_k} \nabla (u_{n_k}-u)\to 0, \quad k\to \infty.
    \end{equation}
  Let's show that \eqref{cond456} implies the existence of a subsequence of  $\{u_{n_k}\}$ which converges strongly in  $\mathcal{W}(\Omega)$.

  We set the operator $S:\mathcal{W}(\Omega) \to [0,\infty)$ as 
  $$S(v)=\frac{1}{2}\|\nabla v\|_{L^2(D_1)}^2+ \frac{1}{p}\|\nabla v\|_{L^p(D_2)}^p,$$
  namely,
  $$
  S(v)=\G(v)+\lambda \int_{\Omega}\widehat{H}(x,u).
  $$
  It is easy to check that $S$ is convex and weakly lower semicontinuous.
  First, we claim that 
  \begin{equation}\label{cond789}
  \lim_{k\to \infty}S(u_{n_k})=S(u).
  \end{equation}

  Indeed, by {\emph{(ii)}} and by the strong convergence of $\{u_{n_k}\}$ in $L^{q+1}(\Omega)$, we get that the sequence $\{S(u_{n_k})\}$ is bounded. Thus, up to a subsequence, $S(u_{n_k})\to a\in \RR$. Moreover, since $S$ is weakly lower semicontinuous, we obtain
  $$
  a=\lim_{k\to \infty} \inf S(u_{n_k})\geq S(u).
  $$
  By the other hand, due to convexity of $S$, i.e.
  $$
  S(u)\geq S(u_{n_k})+S^{\, \prime}(u_{n_k})(u-u_{u_k})
  $$
  and keeping in mind, by \eqref{cond456}, that  $S^{\, \prime}(u_{n_k})(u-u_{u_k}) \to 0$, we obtain (taking limits)
  $$
  S(u)\geq a
  $$
  and the claim \eqref{cond789} is proved.

    Then, to show that there exists a subsequence of $\{u_{n_k}\}$ which converges  strongly to $u$ in $\mathcal{W}(\Omega)$, we argue by contradiction.  We consider a subsequence $\{u_{n_{k_l}}\}$ and $\delta >0$ such that $[u_{n_{k_l}}-u]_{\w}\geq \delta$. In particular, there is a $\tilde \delta>0$ such that $S(u_{n_{k_l}}-u)\geq \tilde \delta$.
    
We have  $$\displaystyle \frac{u_{n_{k_l}}+u}{2}\rightharpoonup u$$ and, by using again that $S$ is weakly lower semicontinuous, it holds
    \begin{equation}\label{parteA}
    S(u)\leq \lim \inf S\left(\displaystyle \frac{u_{n_{k_l}}+u}{2}  \right).
    \end{equation}
    On the other hand, due to Clarkson's inequality:
    $$
    \left|\frac{z+w}{2}  \right|^r+\left|\frac{z-w}{2}  \right|^r\leq \frac{1}{2} |z|^r+\frac{1}{2}|w|^r,\quad z,w\in \RR,\, 2\leq r <\infty.
    $$
    it is easy to check that
    \begin{align*}
    S\left(\displaystyle \frac{u_{n_{k_l}}+u}{2}  \right) &\leq \frac{1}{2}S(u_{n_{k_l}})+\frac{1}{2}S(u)-   S\left(\displaystyle \frac{u_{n_{k_l}}-u}{2}  \right)
    \\
    &\leq \ \frac{1}{2}S(u_{n_{k_l}})+\frac{1}{2}S(u)-\frac{\tilde \delta}{2^p}.
    \end{align*}
    Finally, taking superior limits and taking into account \eqref{cond789}, we have 
    $$
    \lim \sup S\left(\displaystyle \frac{u_{n_{k_l}}+u}{2}  \right)\leq S(u)-\frac{\tilde \delta}{2^p}
    $$
    which, together with \eqref{parteA},  leads to the following contradiction
    $$
    S(u)\leq \lim \inf S\left(\displaystyle \frac{u_{n_{k_l}}+u}{2}  \right)\leq  \lim \sup S\left(\displaystyle \frac{u_{n_{k_l}}+u}{2}  \right)\leq S(u)-\frac{\tilde \delta}{2^p}.
    $$
 \end{proof} 

Now we are ready to find a second solution.

\begin{proof}[Proof of Theorem \ref{Teo2}]
For every fixed $\lambda \in (0,\lambda^*)$, we consider
\[
\Gamma(\lambda):=\{\gamma \in \mathcal{C}([0,1],\w)\, : \, \gamma(0)=\tilde u_{\lambda},\, \gamma(1)=Tw   \}.
\]
Here $\tilde u_{\lambda}$ is the local minimum of the functional $G_{\lambda}$ obtained in Theorem \ref{minimum}. In addition, by construction,  $\tilde u_\lambda$ is greater that $u_1$, the minimal positive solution for $0<\lambda_1<\lambda$ obtained in Theorem \ref{Teo1}. Therefore, $\tilde u_\lambda$ is also a local minimum from $\G$. On the other hand, $0< w\in \mathcal{C}^\infty_c(D_1)$ and $T=T(\lambda)>0$ big enough  to ensure that $T w>u_1$ in $D_1$ and $\G(\tilde u_\lambda)>\G (Tw)$.

  Let's also consider
  \begin{equation*}
c(\lambda):=\inf_{\gamma \in \Gamma(\lambda)} \max_{t\in [0,1]} \G(\gamma(t)).
\end{equation*}
Obviously, $c(\lambda)\geq \max \{\G(\tilde u_\lambda), \G(Tw)   \}= \G(\tilde u_\lambda)=G_\lambda(\tilde u_\lambda)$. Where in the last equality we have used the fact that $u_1<\tilde u_\lambda$.

We distinguish between two possible cases:

\underline{If $c(\lambda)=  \G(\tilde u_\lambda)$}. In this case,  since $\tilde u_\lambda$ is a local minimizer of $\G$, there is $\delta >0$ such that $\G (\tilde u_\lambda)\leq \G(v)$ for all $v$ belongs in the ball  $B_{\delta}(\tilde u_\lambda)=\{v \in \mathcal{W}(\Omega)\,:\, [v-\tilde u_\lambda]_{\w}<\delta \}$. In the case that there is a $v_0\in B_{\delta}(\tilde u_\lambda)\setminus \{\tilde u_\lambda\}$  with $\G (\tilde u_\lambda) = \G(v_0)$, then $v_0$ will be another minimum (in fact, there will be infinity many minimums) and the proof is finished. Therefore, we can suppose 
\[
\G (\tilde u_\lambda)<\G(v), \qquad \forall v\in B_{\delta}(\tilde u_\lambda)\setminus \{\tilde u_\lambda\}.
\]
In particular, for all $r\in (0,\delta)$, it holds
\[
c(\lambda)=\G (\tilde u_\lambda)< \G(v), \qquad \hbox{ if } [\tilde u_\lambda-v]_{\w}=r.
\]
Then, applying  the refinement of the Mountain Pass Theorem dues to Ghoussoub-Preiss \cite[Theorem 1]{G-P} with the closed subset $$F_r=\{v \in \mathcal{W}(\Omega)\,:\, [v-\tilde u_\lambda]_{\w}=r\}\subset \w$$ we obtain the existence of a sequence $\{u_n\}\subset \w$ verifying: $$\lim_n \hbox{dist}(u_n,F_r)=0, \qquad
\lim_n \G(u_n)=c(\lambda)\qquad \mbox{and} \qquad \lim_n \|\G^{\, \prime}(u_n)\|_{\w'}=0.$$ Then, $\{u_n\}$ is bounded (because $F_r$ is bounded and the distance of $u_n$ to $F_r$ goes to zero) and by Lemma \ref{bounded-PS} our functional satisfies the Palais-Smale condition for bounded sequences. Consequently, there exists a critical point of $\G$ on $F_r$ with critical value $c(\lambda)$ (see \cite[Theorem 1. bis]{G-P}). Then, this critical point is a nontrivial weak solution to our problem \eqref{alexis}  (that is in fact strictly greater than $u_1$).
Note that we can apply this reasoning for every closed subset $F_r$ with $r\in (0,\delta)$, and to conclude the existence of infinite critical points of $G_\lambda$  in $B_{\delta}(\tilde u_\lambda)$.

\medskip

\underline{If $c(\lambda)>  \G(\tilde u_\lambda)$}, for some $\lambda=\hlambda \in (0,\lambda^*)$. Let $\lambda_1<\hlambda$ and $u_1$ the  minimal solution in the construction of $\widehat{G}_\hlambda$ in \eqref{hatG}.
In this way, we consider the interval $[\hlambda-\varepsilon_0,\hlambda]$,
with $\varepsilon_0>0$ such that
\begin{equation*}
\varepsilon_0<\min \left\{\frac{(q+1)\varepsilon_1}{\|\tilde u_\hlambda\|_{L^{q+1}(\Omega}^{q+1}}, \hlambda-\lambda_1   \right\},
\end{equation*}
where $\varepsilon_1=c(\hlambda)-\widehat{G}_\hlambda(\tilde u_\hlambda)>0$.
Obviously, $[\hlambda-\varepsilon_0,\hlambda]\subset (0,\lambda^*)$ since $\varepsilon_0<\hlambda$. Then, for this $(u_1,\lambda_1)$ fixed, we define $\G$ for $\lambda \in [\hlambda-\varepsilon_0,\hlambda]$. Of course, $\G$  is non-increasing with respect to $\lambda$. Furthermore, we get for every $\lambda \in [\hlambda-\varepsilon_0,\hlambda]$ :
\begin{align*}
c(\lambda)\geq c(\hlambda) &=\widehat{G}_\hlambda(\tilde u_\hlambda)+\varepsilon_1
\\
&=\widehat{G}_{\hlambda-\varepsilon_0}(\tilde u_\hlambda)+\varepsilon_1-\frac{\varepsilon_0}{q+1}\int_\Omega \tilde u_\hlambda^{q+1}
\\
&>\widehat{G}_{\hlambda-\varepsilon_0}(\tilde u_\hlambda)
\\
&\geq \widehat{G}_{\lambda}(\tilde u_\hlambda),
\end{align*}
where we have used the fact that $\G(\tilde u_\hlambda)=G_\lambda (\tilde u_\hlambda)$ for $\lambda \in [\hlambda-\varepsilon_0,\hlambda]$.

  Summarizing, we have
\begin{equation*}
c(\lambda)>\max \{\widehat{G}_{\lambda}(\tilde u_\hlambda), \widehat{G}_{\lambda}(Tw)\}, \qquad \hbox{ for all } \lambda \in [\hlambda-\varepsilon_0,\hlambda].
\end{equation*}

Finally, applying Jeanjean's result \cite[Theorem 1.1]{Jeanjean}, there exists a bounded Palais-Smale 
sequence at the level $c(\lambda)$ for almost every $\lambda \in [\hlambda-\varepsilon_0,\hlambda]$. This 
Palais-Smale sequence, due Lemma \ref{bounded-PS}, has a subsequence that converges strongly. In this setting, 
by the Mountain Pass theorem due to Ambrosetti and Rabinowitz (\cite{AR}) there exists a critical point of $\widehat{G}_\lambda$ at level $c(\lambda)$ (hence 
different from the minimum $\tilde u_\lambda$) for almost every $\lambda \in [\hlambda-\varepsilon_0,\hlambda]$. Arguing as in the previous case, we obtain a  positive critical point of $G_\lambda$. 

Then, we conclude that there exists  a second positive solution of problem \eqref{alexis} for almost every $\lambda \in (0,\lambda^*)$.
\end{proof}

\section*{Appendix}

We include here a proof of the fact that Palais-Smale sequences are bounded when we assume an
Ambrosetti-Rabinowitz type condition with $\kappa>p$. We remark again that this condition
does not hold here, but we include this simple computation for the sake of completeness.

\begin{lemma}\label{AR-type}
  Consider the functional $F:\w \to \RR$ defined as follows:
\begin{equation*}
F(u)= \int_{D_1} \frac{|\nabla u|^2}{2} \, dx +
\int_{D_2} \frac{|\nabla u|^p}{p} \, dx - \lambda \int_{\Omega} H(x,u(x)) \, dx,
\end{equation*}
with $H$ such that there exists $\kappa>p$ satisfying 
\begin{equation}\label{condition-AR}
0\leq \kappa H(x,s)\leq s h(x,s), \qquad s \geq 0, \, x \in \Omega,
\end{equation}
where $H(x,s)=\int_0^sh(x,t)dt$.

   Then,  Palais-Smale sequences for $F$ are bounded. 
\end{lemma}
\begin{proof}
Let $\{u_n\}\subset \w$ be a Palais-Smale sequence. That is, $|F(u_n)|\leq C$ and $F'(u_n)\to 0$ in $\w'$. Then
\begin{align*}
C & \geq \int_{D_1} \frac{|\nabla u_n|^2}{2} + \int_{D_2} \frac{|\nabla u_n|^p}{p} - \lambda \int_{\Omega} H(x,u_n) \, dx,
\\
& \geq \int_{D_1} \frac{|\nabla u_n|^2}{2} + \int_{D_2} \frac{|\nabla u_n|^p}{p} -\frac{\lambda}{\kappa}\int_{\Omega}u_n h(x,u_n)dx
\\
&=\left(\frac{1}{2}-\frac{1}{\kappa} \right) \int_{D_1} |\nabla u_n|^2 + \left(\frac{1}{p}-\frac{1}{\kappa} \right) \int_{D_2} |\nabla u_n|^p+\frac{1}{\kappa} F'(u_n)(u_n)
\\
& \geq \left(\frac{1}{p}-\frac{1}{\kappa} \right)\left(\int_{D_1} |\nabla u_n|^2+\int_{D_2} |\nabla u_n|^p   \right)-\frac{\varepsilon_n}{\kappa}[u_n]_{\w},
\end{align*}
where $\varepsilon_n \to 0$. This leads to the boundedness of $\{u_n\}$ in $\w$.
\end{proof}
We remark that the condition \eqref{condition-AR} can be relaxed imposing the inequality for $|s|\geq R>0$.

\section*{Ackonwledgement}
This work was started during a research stay of the first author at Universidad of Buenos Aires (Argentina) supported by Secretar\'ia de Estado de Investigaci\'on, Desarrollo e Innovaci\'on EEBB2014 (Spain) also he is partially supported by MINECO-FEDER
Grant MTM2015-68210-P (Spain), Junta de Andaluc\'ia FQM-116 (Spain) and by MINECO Grant BES-2013-066595 (Spain). The second author is supported by CONICET (Argentina) and by MINECO-FEDER Grant MTM2015-70227-P (Spain).

\end{document}